\documentclass[12pt]{amsart}
\usepackage{amsfonts}
\usepackage{amsfonts,latexsym,rawfonts,amsmath,amssymb,amsthm}
\usepackage{amsmath,amssymb,amsthm,amscd,esint}
\usepackage[plainpages=false]{hyperref}
\usepackage{appendix}

\usepackage{graphicx} 
\usepackage{float} 
\usepackage{subfigure} 

\RequirePackage{color}

\numberwithin{equation}{section}

\newcommand{\beq}{\begin{equation}}
\newcommand{\eeq}{\end{equation}}
\newcommand{\beqs}{\begin{eqnarray*}}
\newcommand{\eeqs}{\end{eqnarray*}}
\newcommand{\beqn}{\begin{eqnarray}}
\newcommand{\eeqn}{\end{eqnarray}}
\newcommand{\beqa}{\begin{array}}
\newcommand{\eeqa}{\end{array}}

  \newcommand{\cE}{{\mathcal E}}
  \newcommand{\F}{{\mathcal F}}
  
\newcommand{\cH}{{\mathcal H}}

  \newcommand{\cM}{{\mathcal M}}

\newcommand{\sq}{{\sqrt{-1}}}
\newcommand{\psh}{\mathcal{PSH}}
\newcommand{\psho}{\mathcal{PSH}_0}
\newcommand{\pshk}{\mathcal{PSH}_k}
\newcommand{\pshko}{\mathcal{PSH}_{k,0}}
\newcommand{\ca}{\text{Cap}}

\newcommand{\bS}{{\mathbb S}}
\newcommand{\C}{{\mathbb C}}
\newcommand{\R}{{\mathbb R}}

\newcommand{\p}{\partial}
\newcommand{\eps}{\varepsilon}
\newcommand{\lam}{\lambda}
\newcommand{\Om}{\Omega}
\newcommand{\pom}{{\p\Om}}

\newcommand{\hu}{\hat u}

\newcommand\tbbint{{-\mkern -16mu\int}}  \newcommand\dbbint{{-\mkern -19mu\int}}   \newcommand\bbint{ {\mathchoice{\dbbint}{\tbbint}{\tbbint}{\tbbint}} }

\newtheorem{prop}{Proposition}[section]
\newtheorem{theo}[prop]{Theorem}
\newtheorem{lem}[prop]{Lemma}

\newtheorem{cor}[prop]{Corollary}
\newtheorem{rem}[prop]{Remark}
\newtheorem{ex}[prop]{Example}
\newtheorem{defi}[prop]{Definition}

\allowdisplaybreaks
\begin{document}

\title{Trace Inequalities, isocapacitary inequalities and regularity of the complex Hessian equations}
\author{Jiaxiang $\text{Wang}^1$, Bin $\text{Zhou}^{2}$}

\address{Jiaxiang Wang: 
Beijing International Center for Mathematical Research, Peking University,
Beijing, 100871, P. R. China}
\email{wangjx\underline{ }manifold@126.com}

\address{Bin Zhou:
	School of Mathematical Sciences, Peking
	University, Beijing 100871, China.} 
\email{bzhou@pku.edu.cn}

\thanks{$1$ Partially  supported by China post-doctoral grant BX2021015.}
\thanks {$2$ Partially supported by National Key R$\&$D Program of China SQ2020YFA0712800 and NSFC  Grant 11822101.}


\begin{abstract}
In this paper, we study the relations between trace inequalities(Sobolev and Moser-Trudinger type), 
isocapacitary inequalities and the regularity of the complex Hessian and Monge-Amp\`ere equations with 
respect to a general positive Borel measure. We obtain a quantitative characterization for these relations through properties of
the capacity minimizing functions.
\end{abstract}

 \maketitle


\baselineskip=16.4pt
\parskip=3pt

\section{introduction}

Sobolev and Moser-Trudinger  type inequalities play an important role in both PDE and geometry. On one hand, these inequalities are widely used in the study of existence and regularity of solutions to partial differential equations. On the other hand, people also discovered that they are equivalent to isoperimetric and isocapacitary inequalities \cite{Ma}. Despite the classical Sobolev and Moser-Trudinger inequalities, the analogous  inequalities for a series of fully nonlinear equations with variational structure have been developed, including both real and complex Hessian equations  \cite{W1, TrW, TiW, AC20}.  In particular, the Moser-Trudinger type inequality for the complex Monge-Amp\`ere equations has been established \cite{BB, C19, GKY, AC19,  WWZ1}. 
The Moser-Trudinger type inequality can also be related to the Skoda integrability of plurisubharmonic functions \cite{DNS, DN, Ka, DMV}.
In this paper, we study the relations between trace inequalities(Sobolev and Moser-Trudinger type), isocapacitary inequalities and the regularity of the complex Hessian and Monge-Amp\`ere equations with 
respect to a general nonnegative Borel measure $\mu$. Our results generalize the classical trace inequalities \cite{AH}.  Here `trace' refers to that $\mu$ lives on a domain $\Omega$  and is a surface measure on a smooth submanifold in $\Om$. The trace and isocapacitary inequalities for the real Hessian equations were obtained by \cite{XZ}.

Let $\Om\subset \C^{n}$ be a pseudoconvex domain with smooth boundary $\p \Om$ and $\omega$ be the K\"ahler form associated to the standard Euclidean metric.  
Let $\pshk(\Om)$ be the $k$-plurisubharmonic functions on $\Om$ and $\pshko(\Om)$ be the set of functions in $\pshko(\Om)$
with vanishing boundary value. Let $\mathcal F_k(\Omega)$ be set of $k$-plurisubharmonic functions which can be decreasingly approximated by functions in $\pshko(\Om)\cap C(\overline{\Om})\cap C^2(\Om)$ \cite{C06}. For $u\in\mathcal F_k(\Omega)$,
denote the $k$-Hessian energy by
$$\mathcal E_k(u)=\int_{\Om}(-u)(dd^cu)^k\wedge\omega^{n-k}$$
and the norm by
$$\|u\|_{\pshko(\Om)}:=\left(\int_{\Om}(-u)(dd^cu)^k\wedge\omega^{n-k}\right)^{\frac{1}{k+1}}.$$
In particular, when $k=n$, we write $\|u\|_{\psho(\Om)}=\|u\|_{\mathcal{PSH}_{n,0}(\Om)}$ for simplicity.

Let $\mu$ be a nonnegative Borel measure with  finite mass  on $\Omega \subset \C^n$.
We will consider the  trace inequalities and 
isocapacitary inequalities with respect to $\mu$ as in the classical case \cite{AH}. We recall the capacity for plurisubharmonic functions.
The relative capacity for plurisubharmonic functions was introduced by Bedford-Taylor \cite{BT, B}.
For a Borel subset $E\subset \Om$, the {\it $k$-capacity} is defined as
$$\ca_k(E,\Om)=\sup\left\{\left.\int_{E}(dd^cv)^k\wedge\omega^{n-k}\,\right|\, v\in\pshk(\Om),-1\le v\le 0\right\}.$$
Throughout the context, we will use $|\cdot|$ to denote the Lebesgue measure of a Borel subset.
We follow \cite{Ma} to define the {\it capacity minimizing function} with respect to $\mu$
$$\nu_k(s,\Omega,\mu):=\inf\left\{\left.\ca_k(K,\Om)\,\right|\,\mu(K)\ge s, K\Subset\Om\right\},\ \ \ 0<s<\mu(\Om).$$
Then we denote
\begin{align}\label{cond-sobolev0}
	I_{k, p}(\Omega,\mu):=\begin{cases}
	\int_0^{\mu(\Om)}\left(\frac{s}{\nu_k(s,\Omega,\mu)}\right)^{\frac{p}{k+1-p}}\,ds, & 0<p<k+1,\\[5pt]
	 \displaystyle\sup_{t>0}\left\{\frac{t}{\nu_k(t,\Omega,\mu)^{\frac{p}{k+1}}}\right\}, &  p\geq k+1,
\end{cases}
	\end{align}
	and
	\begin{align}\label{cond-MT}
I_{n}(\beta,\Omega,\mu):=\sup\left\{s\exp\left(\frac{\beta}{\nu_n(s,\Omega,\mu)^{\frac{q}{n+1}}}\right):
\ 0<s<\mu(\Om)\right\}.
\end{align}
Note that if $\mu(P)>0$ for a $k$-pluripolar subset $P\subset \Om$, then $I_{k, p}(\Omega,\mu)$, $I_{n}(\beta,\Omega,\mu)=+\infty$. Therefore, we only need to consider those measures which charge no mass on pluripolar subsets.
The main result of this paper is as follows.

\begin{theo}\label{main3}
	Suppose $\Om\subset\C^n$ is a smooth, $k$-pseudoconvex domain, where $1\le k\le n$. 
 \begin{enumerate}
\item[(i)] The Sobolev type trace inequality   
	\begin{align}\label{sobolev1}
	\sup\left\{\frac{\|u\|_{L^p(\Om,\mu)}}{\|u\|_{\pshko(\Om)}}: u\in \mathcal F_k(\Omega), 0<\|u\|_{\pshko(\Om)}<\infty\right\}<+\infty
	\end{align}
	holds if and only if $I_{k, p}(\Omega,\mu)<+\infty$. Moreover,  the Sobolev trace map 
$$Id:\,\F_k(\Om)\hookrightarrow L^p(\Om,\mu), \ 1<p<\infty$$ is a compact embedding if and only if
\begin{align}\label{cond-compact}
\lim_{s\to 0}\frac{s}{\nu_k(s,\Omega,\mu)^{\frac{p}{k+1}}}\to 0.
\end{align}

\item[(ii)] When $k=n$, for $q\in [1,\frac{n+1}{n}]$ and $\beta>0$, the Moser-Trudinger type trace inequality 
\begin{align}\label{MT}
\sup\left\{\int_{\Om}\exp\left(\beta\left(\frac{-u}{\|u\|_{\psho(\Om)}}\right)^{q}\right)\,d\mu: u\in \mathcal F_n(\Omega), 0<\|u\|_{\psho(\Om)}<\infty\right\}
\end{align}
holds if and only if $I_{n}(\beta,\Omega,\mu)<+\infty$.

\end{enumerate}
\end{theo}

\begin{rem}\label{rem}
(a) For $p\geq k+1$, $\beta>0$, the conditions $I_{k, p}(\Omega,\mu), I_{n}(\beta,\Omega,\mu)<+\infty$ are equivalent to the following
 isocapacitary type inequalities
\begin{eqnarray}
&&\mu(K)\le I_{k, p}(\Omega,\mu)\cdot\ca_k(K,\Om)^{\frac{p}{k+1}}, \label{cond-sobolev1}\\
&&\mu(K)\leq I_{n}(\beta,\Omega,\mu)\cdot\exp\left(-\frac{\beta}{\ca_n(K,\Om)^{\frac{q}{n+1}}}\right) \label{cond-MT1}
\end{eqnarray}
for $K\Subset \Om$.
By \cite{K96, DK},
it is known that when $\mu$ is Lebesgue measure
\begin{eqnarray}
|E|&\leq& C_{\lambda,\Om}\cdot\ca_k(E,\Om)^{\lam},\ \lam<\frac{n}{n-k},\ 1\le k\le n-1,\label{isocap1}\\[3pt]
|E|&\leq& C_{\beta,n,\Om}\cdot \exp\left(-\frac{\beta}{\ca_n(E,\Om)^{\frac{1}{n}}}\right),\  0<\beta<2n. \label{isocap}
\end{eqnarray}
It is still open whether $\beta$ can attain $2n$. 
By a result in \cite{BB}, the conclusion is true for  subsets $E\Subset\Om$ with $\bS^1$-symmetry(invariant under the rotation $e^{\sq\theta}z$ for all $\theta\in \R^1$), when $\Om\subset \C^n$ is a ball centered at the origin. See Remark \ref{rem1} for more explanations.

(b) By the arguments of \cite[Section 5]{BB}, \eqref{MT} is equivalent to   
\begin{align*}
\sup\left\{\int_{\Om}\exp\left(k(-u)-\frac{n^nk^{1+n}\cE_n(u)}{(n+1)^{1+n}\beta^n}\right)^{q}\,d\mu: \forall k>0, u\in \mathcal F_n(\Omega)\right\}<+\infty.
\end{align*}

(c) We can also prove the quasi Moser-Trudinger type trace inequality is equivalent to the quasi  Brezis-Merle type trace inequality(see Theorem \ref{remB}).  
Note that the proof  for the case $\mu$ is Lebesgue measure in \cite{BB} used thermodynamical formalism and a dimension induction argument. Our proof here uses the isocapacitary inequality \eqref{cond-MT1}. 
\end{rem}

\begin{ex}
Let $\mu$ be the measure with singularities of Poincar\`e type, i.e.,
$$d\mu=\frac{1}{\prod_{j=1}^d|z^j|^2\left(1-\log |z^j|\right)^{1+\alpha}}\,dz^1\wedge d\bar z^1\wedge\cdots\wedge dz^n\wedge d\bar z^n,  \ \ \ \text{where }d\le n, \ \alpha>0.$$
nn the unit disk $\mathbb D^n\subset \mathbb C^n$, 
	According to \cite[Lemma 4.1]{DL}, we have
	$$\mu(K)\le C\cdot {\ca_n(K,\mathbb D^n)}^{\alpha}.$$
By Theorem \ref{main3},  a Sobolev type inequality with respect to $d\mu$ holds for plurisubharmonic functions.
\end{ex}

Now we turn to the corresponding equations.  
Consider the Dirichlet problem
	\begin{equation}\label{cMA}
\begin{cases}
(dd^cu)^k\wedge\omega^{n-k}=d\mu   \ \ &\text{in\ $\Omega$,}  \\
u=\varphi,   &\text{on }\p\Om
\end{cases}
\end{equation}  
for  a nonnegative Borel measure $\mu$.
In a seminal work \cite{K98},
Ko\l{}odziej obtained the $L^\infty$-estimate and existence of continuous solutions for the complex Monge-Amp\`ere equation when $d\mu$ is dominated by a suitable function of capacity, especially for $d\mu\in L^p(\Omega)$.
Furthermore, the solution  is shown to be H\"older continuous under certain assumptions on $\Omega$ and $\varphi$ \cite{GKZ}. 
Generalizations to the complex Hessian equations were made by \cite{DK, Ngu}. 
These results were established by pluripotential theory. 
In \cite{WWZ2}, the authors present a new PDE proof for the complex Monge-Amp\`ere equation with $d\mu\in L^p(\Omega)$ based on the Moser-Trudinger type inequality. 
By Theorem \ref{main3}, we have

\begin{theo}\label{main5}
Let $\mu$ be a non-pluripolar, nonnegative Radon measure with finite mass. Then the following statements are equivalent:
\begin{enumerate}
\item [(i)] There exist $0<\delta<\frac{1}{k}$ and a  constant $C>0$ depending on $\mu$ and $\Omega$ such that for any Borel subset $E\subset \Om$, 
the Dirichlet problem
\begin{align}\label{k-Dir}
\begin{cases}
(dd^cu)^k\wedge\omega^{n-k}=\chi_E\,d\mu,\ \ \ &\text{in }\Om,  \\
u=0,   &\text{on }\p \Om,
\end{cases}
\end{align}
 admits a continuous solution $u_E\in \pshko(\Om)$ such that 
\begin{align}\label{infty}
\|u_E\|_{L^{\infty}(\Om,\mu)}\le C\mu(E)^{\delta}.
\end{align}
Here $\chi_E$ is the characteristic function of $E$.

\item [(ii)] 
There exists  $p\ge k+1$ such that  
$I_{k,p}(\Om,\mu)<+\infty$.

\end{enumerate}

More precisely,  $\delta$ and $p$ can be determined mutually by $p=\frac{k+1}{1-k\delta}$.
\end{theo}

\begin{rem}
(1) The conclusion from (ii) to (i) in the above theorem also holds with general continuous boundary value.

(2) When $d\mu$ is an integrable function, once we have the $L^{\infty}$-estimate for  complex  Monge-Amp\`ere equation,  the $L^{\infty}$-estimate of 
many other equations, including the complex Hessian equations  and $p$-Monge-Amp\`ere operations \cite{HL09}, can be derived by a simple comparison. In real case, this is indicated by  \cite{W2}. For example, we consider the
complex $k$-Hessian equation
\beq\label{eqhe}
\begin{cases}
(dd^cu)^k\wedge \omega^{n-k}=f\,\omega^n,\ \ \ &\text{in }\Om,  \\
u=\varphi,   &\text{on }\p \Om.
\end{cases}
\eeq
Suppose $d\mu=f\,\omega^n$ with $f\in L^{\frac{n}{k}}(\log L^1)^{n+\eps}$. Let $v$ be the solution to the complex Monge-Amper\`e equation 
$$
\begin{cases}
(dd^cv)^n=f^{\frac{n}{k}}\,\omega^n,\ \ \ &\text{in }\Om,  \\
v=\varphi,   &\text{on }\p \Om.
\end{cases}
$$ 
By elementary inequalities, $v$ is a subsolution to \eqref{eqhe}. Then $\|u\|_{L^\infty(\Omega)}\leq \|v\|_{L^\infty(\Omega)}\leq C$.
However, when $\mu$ is a general measure, this comparison does not work.
\end{rem}

It is also interesting to ask when the solution is H\"older continuous.
In \cite{DKN}, Dinh, Ko\l{}odziej and Nguyen introduced a new condition on $d\mu$ and proved  that it is equivalent to the H\"older continuity for the complex Monge-Amp\`ere equation. We give a pure PDE proof as well as for the complex Hessian equations, based on the Sobolev type inequality for complex Hessian operators and the arguments in \cite{WWZ2}.
 
 As in \cite{DKN}, we denote by $W^*(\Om)$  the set of functions $f\in W^{1,2}(\Om)$ such that 
$$df\wedge d^cf\le T$$
for some closed positive $(1,1)$-current $T$ of finite mass  on $\Om$. Define a Banach norm by
$$\|f\|_*:=\|f\|_{L^1(\Om)}+\min\left\{\|T\|_{\Om}^{\frac{1}{2}}\,\big|\,T\text{ as above}\right\},$$
where the mass of $T$ is defined by $\|T\|_{\Om}:=\int_\Om T\wedge\omega^{n-1}$.
\begin{theo}\label{main6}
Suppose $1\le k\le n$, $\Om$ is a $k$-pseudoconvex domain with smooth boundary. Let $\mu$ be a Radon measure with finite mass.  Let $\gamma\in\left(\frac{(n-k)(k+1)}{2nk+n+k}, k+1\right]$. 
The following statements are equivalent:
\begin{enumerate}
	\item [(i)]The Dirichlet problem \eqref{cMA} admits a solution $u\in C^{0,\gamma'}(\Om)$ with $$0<\gamma'<\frac{(2nk+n+k)\gamma-(n-k)(k+1)}{(n+1)k\gamma+(n+1)k^2+nk+k}.$$

	\item [(ii)] There exists  $C>0$ such that for every smooth function $f\in W^*(\Om)$ with $\|f\|_*\le 1$,
\begin{align}\label{cond-DKN}
\mu(f):=\int_\Omega f\,d\mu\le C\|f\|_{L^1(\Om)}^{\gamma}.
\end{align}
\end{enumerate}
\end{theo}

The structure of the paper is as follows: Section \ref{pre} is devoted to a review on the relative capacity for $k$-plurisubharmonic functions. In particular, we obtain several equivalent definitions for the capacity. Section \ref{capest}, we establish the capacitary estimates for level sets of $k$-plurisubharmonic functions, which is the main tool in the proof of Theorem \ref{main3}. Theorem \ref{main3} will be proved case by case in Section \ref{mainpf}.  In the last section, we apply Theorem \ref{main3} to the Dirichlet problem for the complex Hessian equations to prove Theorem \ref{main5} and \ref{main6}.

\vskip 20pt

\section{On relatively capacities and relatively extremal functions}\label{pre}

In this section, we recall the relative capacity for plurisubharmonic functions \cite{BT, B}. 

For a Borel subset $E\subset \Om$, the {\it $k$-capacity} is defined as
$$\ca_k(E,\Om)=\sup\left\{\left.\int_{E}(dd^cv)^k\wedge\omega^{n-k}\,\right|\, v\in\pshk(\Om),-1\le v\le 0\right\}.$$
It is well known that the $k$-capacity can be characterized by
the {\it relatively $k$-extremal function} $u_{k,E,\Omega}^*$, which is the upper regularization of 
$$u_{k,E,\Omega}:=\sup\left\{v\,\left|v\in \pshk(\Om),v\le -1\text{ on }E, v<0\text{ on }\Om\right.\right\}.$$
We will usually write  $u_{k,E}^*$ for simplicity if there is no confusion.
When $E=K$ is a compact subset, we have $-1\le u_{k,K}^*\le 0$, the complex Hessian measure $(dd^cu_K^*)^k\wedge\omega^{n-k}=0$ on  $\Om\setminus K$. Moreover, $\ca_k(K,\Om)=0$ if $u_{k,K}^*>-1$ on $K$. 
The following well-known fact shows $u_{k,K}^*\in\pshko(\Om)$.
\begin{lem}\label{exb}
If $K\subset \Om$ is a compact subset,
 we have
$u_{k,K}^*\big|_{\p \Om}=0$. 
\end{lem}
\begin{proof}	
By \cite[Proposition 1.2]{KR}, there exists an exhaustion function $\psi\in C^{\infty}(\Om)\cap \pshko(\Om)$.
By the maximum principle we have $-a:=\sup_K\psi<0$. Let $\hat \psi:=\frac{\psi}{a}\le -\chi_K$, then we get $\hat\psi\le u_{k,K}^*$ on $\Om$. By the fact $\hat\psi(\xi)\to 0$ as $\xi\to z\in \p\Om$, we get the result.
\end{proof}

Suppose $\Om$ is $k$-hyperconvex so that $\pshko(\Om)$ is non-empty. 
Inspired by \cite{XZ} for the real Hessian equations, we consider several capacities defined as follows. 

\begin{defi}\label{cap-2}
(i) Let $K$ be a compact subset of $\Om$. Define
\begin{align}
&\widetilde{\ca}_{k,1}(K,\Om)=\sup\left\{\left.\int_K (-v)(dd^cv)^k\wedge\omega^{n-k}\,\right|\, v\in \pshko(\Om), -1\le v\le 0\right\}
, \\
&\widetilde{\ca}_{k,2}(K,\Om):=\inf\left\{\left.\int_{\Om}(dd^cv)^k\wedge\omega^{n-k}\,\right|\,v\in \pshko(\Om), v\big|_{K}\le -1\right\},\\
&\widetilde{\ca}_{k,3}(K,\Om):=\inf\left\{\left.\int_{\Om}(-v)(dd^cv)^k\wedge\omega^{n-k}\,\right|\,v\in \pshko(\Om), v\big|_{K}\le -1\right\}.
\end{align}

(ii) For an open subset $O\subset \Om$ and $j=1$, $2$, $3$, let
$$\widetilde{\ca}_{k,j}(O,\Om):=\sup\left\{\left.\widetilde{\ca}_{k,j}(K,\Om)\,\right|\,\text{compact }K\subset O\right\}.$$

(iii) For a Borel subset $E\subset \Om$ and $j=1$, $2$, $3$, let
$$\widetilde{\ca}_{k,j}(E,\Om):=\inf\left\{\left.\widetilde{\ca}_{k,j}(O,\Om)\,\right|\,\text{open $O$ with }E\subset O\subset \Om\right\}.$$
\end{defi}

In order to show the equivalence of capacities defined above, we need the following well-known comparison principle.
\begin{lem}\label{compar}
	Suppose $\Om$ is a $k$-hyperconvex domain with $C^1$-boundary. 
	Let $u$, $v\in \F_k(\Om)$. If $\lim_{z\to \xi\in \p \Om}(u-v)\ge 0$, $u\le v$ in $\Om$, then there hold
	$$\int_{\Om}(dd^cu)^k\wedge\omega^{n-k}\ge \int_{\Om}(dd^cv)^k\wedge\omega^{n-k},\ \ \ \int_{\Om}(-u)(dd^cu)^k\wedge\omega^{n-k}\ge \int_{\Om}(-v)(dd^cv)^k\wedge\omega^{n-k}.$$
\end{lem}
\begin{proof}
	It is a direct consequence of the integration by parts and the smooth approximation for functions in $\F_k(\Om)$.
\end{proof}

\begin{lem}\label{cap-extre}
	Suppose $1\le k\le n$, and $K\subset \Om$ is a compact subset. Then
\begin{align}
\widetilde{\ca}_{k,1}(K,\Om)=\int_K(-u_{k,K}^*)(dd^cu_{k,K}^*)^k\wedge\omega^{n-k}.
\end{align}
\end{lem}
\begin{proof}
First, by definition we have 
$$\widetilde{\ca}_{k,1}(K,\Om)\ge \int_K(-u_{k,K}^*)(dd^cu_{k,K}^*)^k\wedge\omega^{n-k}.$$ 

To reach the reversed inequality, 
we choose $\{K_j\}$ to be a sequence of compact subsets of $\Om$ with smooth boundaries $\p K_j$ such that
$$K_{j+1}\subset K_j,\ \ \ \bigcap_{j=1}^{\infty}K_j=K.$$
Using the smoothness of $\p K_j$,  the relatively extremal function
$u_j:=u_{k,K_j}^*=u_{k,K_j}\in C(\overline \Om)$.
Note that $u_j\uparrow v$ and $u_{k,K}^*=v^*$.
Then for  $u\in \pshko(\Om)$ such that $-1\le u\le 0$, we have $u\ge u_j$ on $K_j$. Therefore, by Lemma \ref{compar},
\begin{eqnarray*}
\int_{K}(-u)(dd^cu)^k\wedge\omega^{n-k}
&\le& \int_{\{u_j\le u\}}(-u)(dd^cu)^k\wedge\omega^{n-k}\\
&\le& \int_{\{u_j\le u\}}(-u_j)(dd^cu_j)^k\wedge\omega^{n-k}  \\
&\le& \int_{\Om}(-u_j)(dd^cu_j)^k\wedge\omega^{n-k}= \int_{K_j}(-u_j)(dd^cu_j)^k\wedge\omega^{n-k}.
\end{eqnarray*}
Taking $j\to \infty$, we obtain $$\int_{K}(-u)(dd^cu)^k\wedge\omega^{n-k}\le \int_{\Om}(-u_{k,K}^*)(dd^cu_{k,K}^*)^k\wedge\omega^{n-k}=\int_{K}(-u_{k,K}^*)(dd^cu_{k,K}^*)^k\wedge\omega^{n-k}.$$
This yields
$$\widetilde{\ca}_{k,1}(K,\Om)\le \int_{K}(-u_K^*)(dd^cu_K^*)^k\wedge\omega^{n-k},$$
thereby completing the proof. 
\end{proof}

\begin{lem}\label{equal}
	For any Borel set $E\subset \Om$, we have
$$\widetilde{\ca}_{k,j}(E,\Om)=\ca_k(E,\Om),\ \ \ j=1,2,3.$$
\end{lem}
\begin{proof}
	By definition, it suffices to prove the equalities when $E=K$ is a compact subset of $\Om$.
Note that for $u$, $v\in \pshko(\Om)$ such that $-1\le u\le 0$, $v\big|_{K}\le -1$, we have
$$\int_K(-u)(dd^cu)^k\wedge\omega^{n-k}\le \int_K(dd^cu)^k\wedge\omega^{n-k}\le  \int_K(dd^cu_{k,K}^*)^k\wedge\omega^{n-k}\le \int_{\Om}(dd^cv)^k\wedge\omega^{n-k}.$$
Hence we obtain $$\widetilde{\ca}_{k,1}(K,\Om)\le \ca_k(K,\Om)\le \widetilde{\ca}_{k,2}(K,\Om).$$ 
It suffices to  prove
\beq\label{capeq}
\widetilde{\ca}_{k,2}(K,\Om)\le \widetilde{\ca}_{k,3}(K,\Om)\le \widetilde{\ca}_{k,1}(K,\Om).
\eeq

We still choose $\{K_j\}$ to be a sequence of compact subsets in $\Om$ with smooth boundaries $\p K_j$ such that
$$K_{j+1}\subset K_j,\ \ \ \bigcap_{j=1}^{\infty}K_j=K.$$
Then $u_j:=u_{k,K_j}^*=u_{k,K_j}\in C(\overline{\Om})$, $u_j\uparrow v$ with $v^*=u_{k,K}^*$, and $u_j\big|_{K_j}\equiv-1$. For any $j$, we have 
\beqs
\widetilde{\ca}_{k,2}(K,\Om)&\le& \int_{\Om}(dd^cu_j)^k\wedge\omega^{n-k}\\
&=&\int_{\Om}(-u_j)(dd^cu_j)^k\wedge\omega^{n-k}\le \int_{\Om}(-v_j)(dd^cv_j)^k\wedge\omega^{n-k},
\eeqs
where $v_j$ is an arbitrary function in $\pshko(\Om)$ such that $v_j\big|_{K_j}\le -1$. This implies $\widetilde{\ca}_{k,2}(K,\Om)\le \widetilde{\ca}_{k,3}(K_j,\Om)$. Then by Lemma \ref{compar}, we have
$$\widetilde{\ca}_{k,3}(K_{j+1},\Om)\le \int_{\Om}(-u_{j})(dd^cu_{j})^k\wedge\omega^{n-k}=\widetilde{\ca}_{k,1}(K_{j},\Om)=\ca_{k}(K_j,\Om).$$
Letting $j\to\infty$, \eqref{capeq} holds by the convergence of $\ca_{k}(K_j,\Om)$ and the weak continuity of $(-u_j)(dd^cu_j)^k\wedge\omega^{n-k}$.
\end{proof}

\vskip 20pt

\section{Capacitary estimates for level sets of plurisubharmonic functions}\label{capest}
In this section, we establish a capacitary estimate for level sets of $k$-plurisubharmonic functions, which will play 
an important role in the proof of Theorem \ref{main3}. The analogous estimate for the real Hessian equation was obtained by \cite{XZ}. This estimate is a generalization of the capacitary estimates for the Wiener capacity \cite[Chaper 7]{AH}.

First, we recall the {\it capacitary weak type inequality}
		\begin{align}\label{w-ineq}
		t^{k+1}\ca_k(\left\{z\in \Om\,\left|\,u(z)\le -t\right.\right\},\Om)\le \|u\|_{\pshko(\Om)}^{k+1},\ \ \ \forall t>0.
		\end{align}
This inequality was proved by \cite{ACKZ} for $k=n$, and by \cite{Lu} for general $k$.  We prove the {\it capacitary strong type inequality} as follows.
\begin{theo}
	Suppose $u\in \pshko(\Om)\cap C^2(\Om)\cap C(\overline{\Om})$. For any $A>1$, we have
		\begin{align}\label{s-ineq}
		\int_0^{\infty}t^k\ca_k(\left\{z\in\Om\,\left|\,u<-t\right.\right\},\Om)\,dt\le \left(\frac{A}{A-1}\right)^{k+1}\log A \cdot\|u\|_{\pshko(\Om)}^{k+1}.
		\end{align}
		\end{theo}

\begin{proof} 
We use similar arguments as \cite{Ma} for the Laplacian and \cite{XZ} for the real Hessian case. 
	For $t>0$, denote 
	$$K_t:=\left\{z\in \Om\,\left|\,u(z)\le -t\right.\right\}, \ \Om_t:=\left\{z\in \Om\,\left|\,u(z)< -t\right.\right\}$$
	 and $v_t:=u_{k,K_t}^*$. For a Borel subset $E\subset \Om$, we denote
	$$\phi(E):=\frac{\int_E(-u)(dd^cu)^k\wedge\omega^{n-k}}{\int_{\Om}(-u)(dd^cu)^k\wedge\omega^{n-k}}.$$
	For $A>1$, 
	\begin{align}
	\int_0^{\infty}\phi(\Om_t\setminus K_{At})\,\frac{dt}{t} \le &  \int_0^{\infty}\phi(K_t\setminus K_{At})\,\frac{dt}{t}  \nonumber \\
	= & \int_0^{\infty}\left(\int_t^{At}\frac{d\phi(K_s)}{ds}\,ds\right)\,\frac{dt}{t}   \nonumber\\
	= & \int_0^{\infty}\left(\int_s^{\frac{s}{A}}\frac{dt}{t}\right)\frac{d\phi(K_s)}{ds}\,ds  \nonumber \\
	= & -\log A\int_0^{\infty}\frac{d\phi(K_s)}{ds}\,ds  \nonumber\\
	= & \lim_{t\to 0^+}\phi(K_t)\cdot\log A   \nonumber \\
	\le & \log A.  \nonumber
	\end{align}
	This implies
	\begin{eqnarray}\label{key-1}
	\int_0^{\infty}\left\|u\cdot \chi_{\Om_t\setminus K_{At}}\right\|_{\pshko(\Om)}^{k+1}\, \frac{dt}{t}&=& \int_{0}^{\infty}\left(\int_{\Om_t\setminus K_{At}}(-u)(dd^cu)^k\wedge\omega^{n-k}\right)\,\frac{dt}{t} \nonumber\\[4pt]
	&\le& \|u\|_{\pshko(\Om)}^{k+1}\log A.
	\end{eqnarray}
	Then for $\forall t>0$, we consider 
	$$u^t:=\frac{u+t}{(A-1)t},\ \ \ \tilde{u}^t:=\max\{u^t, -1\}.$$
It is clear that	$u^t,\tilde u^t\in \pshko(\Om_t)\cap C^{0,1}(\overline\Om_t)$ and $\tilde u^t= -1$ on $K_{At}$.
We have
	\begin{align*}
	\|\tilde u^t\|_{\pshko(\Om_t)}^{k+1}= & \int_{\Om_t}(-\tilde{u}^t)(dd^c\tilde{u}^t)^k\wedge\omega^{n-k}  \\
	= & \int_{\Om_t}d\tilde{u}^t\wedge d^c\tilde{u}^t\wedge (dd^c\tilde{u}^t)^{k-1}\wedge\omega^{n-k}  \\
	= & \int_{\Om_t\setminus K_{At}}d\tilde{u}^t\wedge d^c\tilde{u}^t\wedge (dd^c\tilde{u}^t)^{k-1}\wedge\omega^{n-k}   \\
	\le & \int_{\Om_t\setminus K_{At}}\left(-\frac{u}{(A-1)t}\right)(dd^c\tilde{u}^t)^{k}\wedge\omega^{n-k} \\
	= & (A-1)^{-k-1}t^{-k-1}\int_{\Om_t\setminus K_{At}}\left(-u\right)(dd^cu)^k\wedge\omega^{n-k},
	\end{align*}
	where we have used $\frac{\p \tilde{u}^t}{\p \nu}\ge 0$ almost everywhere on $\p K_{At}$ for $t>0$ at the fourth line. That is,
	\begin{eqnarray*}
\int_{\Om_t}(-\tilde{u}^t)(dd^c\tilde{u}^t)^k\wedge\omega^{n-k}\le (A-1)^{-k-1}t^{-k-1}\phi(\Om_t\setminus K_{At})\cdot \|u\|_{\pshko(\Om)}^{k+1}.
	\end{eqnarray*}
	Now we denote by $\hat v_t:=u_{k,K_{At},\Om_t}^*$ the relatively extremal function of $K_{At}$ with respect to $\Om_t$.  Note that $\hat v_t\ge \tilde{u}^t$ in $\Om_t$, and $\hat v_t=\tilde{u}^t=0$ on $\p \Om_t$. By comparison principle, we have
	\begin{align}\label{key-2}
	\ca_k(K_{At},\Om_t)=  & \int_{K_{At}}(-\hat v_t)(dd^c\hat v_t)^k\wedge\omega^{n-k}  
	\le  \int_{\Om_t}(-\tilde{u}^t)(dd^c\tilde{u}^t)^k\wedge\omega^{n-k}  \nonumber\\
	\le & (A-1)^{-k-1}t^{-k-1}\phi(\Om_t\setminus K_{At})\cdot \|u\|_{\pshko(\Om)}^{k+1}.
	\end{align}
	Finally, by \eqref{key-1}, \eqref{key-2} with $\lambda=At$, we obtain
	\begin{eqnarray*}
	\int_0^{\infty}\lambda^k \ca_k(K_\lambda,\Om)\,d\lambda &\le & A^{k+1}\int_0^{\infty}t^k \ca_k(K_{At},\Om_t)\,dt  \\
	&\le & A^{k+1}(A-1)^{-k-1}\log A \cdot\|u\|_{\pshko(\Om)}^{k+1}.
	\end{eqnarray*}
\end{proof}

\vskip 10pt

\section{The Trace Inequalities}\label{mainpf}

 In this section, we are going to prove the trace inequalities  in Theorem \ref{main3}.
  \subsection{The Sobolev type trace inequality(the case $0<p<k+1$).} First, we show $I_{k,p}(\Omega,\mu)<+\infty$
  implies the Sobolev type inequality.
For $u\in \mathcal F_k(\Om)$, denote 
 $$S_{k,p}(\mu,u):=\sum_{j=-\infty}^{\infty}\frac{[\mu(K_{2^j}^u)-\mu(K_{2^{j+1}}^u)]^{\frac{k+1}{k+1-p}}}{\ca_k(K_{2^j}^u,\Omega)^{\frac{p}{k+1-p}}},\ \ \ K_s^u:=\{u\le -s\}.$$
By the elementary inequality $a^c+b^c\le (a+b)^c$ for $a$, $b\ge 0$ and $c\ge 1$, we get
\begin{align*}
S_{k,p}(\mu,u) 
\le & \sum_{j=-\infty}^{\infty}\frac{[\mu(K_{2^j}^u)-\mu(K_{2^{j+1}}^u)]^{\frac{k+1}{k+1-p}}}{[\nu_k(2^j,\Omega,\mu)]^{\frac{p}{k+1-p}}}  \\
\le & \sum_{j=-\infty}^{\infty}\frac{[\mu(K_{2^j}^u)]^{\frac{k+1}{k+1-p}}-[\mu(K_{2^{j+1}}^u)]^{\frac{k+1}{k+1-p}}}{[\nu_k(2^j,\Omega,\mu)]^{\frac{p}{k+1-p}}}  \\
\le & \int_0^{\infty}\frac{1}{[\nu_k(s,\Omega,\mu)]^{\frac{p}{k+1-p}}}\,d\left(s^{\frac{k+1}{k+1-p}}\right)  
=  \frac{k+1}{k+1-p}I_{k,p}(\Omega,\mu).
\end{align*}
Then by the strong capacitary inequality \eqref{s-ineq} with $A=n$ and by integration by parts, 
\begin{align*}
\int_{\Om}(-u)^p\,d\mu  
= & \int_{\Om}\int_0^{\infty}\chi_{[0,-u]}(s^p)\,d\left(s^p\right)\,d\mu  \\
= & \int_0^{\infty}\mu(K_s^u)d\left(s^p\right)  \\
= & \int_0^{\infty} s^p\,d\mu(K_s^u)  \\
\le & \sum_{j=-\infty}^{\infty}2^{(j+1)p}\cdot [\mu(K_{2^j}^u)-\mu(K_{2^{j+1}}^u)]  \\
\le & {S_{k,p}(\mu,u)}^{\frac{k+1-p}{k+1}}\cdot\left(\sum_{j=-\infty}^{\infty}2^{j(k+1)}\ca_k(K_{2^{j(k+1)}}^u,\Om)\right)^{\frac{p}{k+1}}  \\
\le & (k+1){S_{k,p}(\mu,u)}^{\frac{k+1-p}{k+1}}\cdot \left(\int_0^{\infty}s^k\ca(K_s^u,\Om)\,ds\right)^{\frac{p}{k+1}}  \\
\le & C(n,k,\mu,p)\cdot\|u\|_{\pshko(\Om)}^p,
\end{align*}
thereby completing the proof.\qed

On the other hand, suppose there exists $p\in(0,k+1)$ such that the inequality \eqref{sobolev1} holds.
That is, for any $u\in\F_k(\Om)$
$$\sup_{t>0}\{t\mu(K_t^u)^{\frac{1}{p}}\}\le \|u\|_{L^p(\Om,\mu)}\le C\|u\|_{\pshko(\Om)}.$$
By the definition of $\nu_k$, for any integer $j<\frac{\log \mu(\Om)}{\log 2}$, we can choose a compact subset $K_j\subset \Om$  such that 
\begin{eqnarray*}
\mu(K_j)\ge 2^j,  \ \ca_k(K_j,\Om)\le 2\nu_k(2^j,\Om,\mu).
\end{eqnarray*}
Then by Lemma \ref{equal}, we can choose  $u_j\in\F_k(\Om)$ such that
\begin{eqnarray*}
u_j\big|_{K_j}\le -1,\
\cE_k(u_j)\le 2\ca_k(K_j,\Om).
\end{eqnarray*}
Then for integers $i$, $m$ with $-\infty< i< m$ and $\frac{\log \mu(\Om)}{\log 2}\in (m,m+1]$, let
$$\gamma_j:=\left(\frac{2^j}{\nu_k(2^j,\Omega,\mu)}\right)^{\frac{1}{k+1-p}}, \ \ \forall j\in [i,m] \text{\ and \ }  u_{i,m}:=\sup_{i\le j\le m}\{\gamma_ju_j\}.$$
Since $u_{i,m}\in \F_k(\Om)\cap C(\overline{\Om})$, we have
$$\cE_k(u_{i,m})\le C_{n,k}\sum_i^m\gamma_j^{k+1}\cE_k(u_j)\le C_{n,k}\sum_i^m\gamma_j^{k+1}\nu_k(2^j,\Omega,\mu).$$
Note that for $i\le j\le m$ we have 
$$2^j<\mu(K_j)\le \mu(K_{\gamma_j}^{u_{i,m}}).$$
This implies
\begin{eqnarray*}
\|u_{i,m}\|^p_{\pshko(\Om)}&\ge & C^{-p}C_{n,k,p}\int_{\Om}|u_{i,m}|^p\,d\mu  \\
&\ge & C\int_0^\infty\left(\inf\left\{t\,\left|\,\mu(K_t^{u_{i,m}})\le s\right.\right\}\right)^p\,ds  \\
&\ge & C\sum_{j=i}^m \left(\inf\left\{t\,\left|\,\mu(K_t^{u_{i,m}})\le 2^j\right.\right\}\right)^p \cdot 2^j  \\
&\ge & C\sum_{j=i}^m\gamma_j^p2^j  \\
&\ge & C\frac{\sum_{j=i}^m\gamma_j^p2^j}{\sum_{j=i}^m\gamma_j^{k+1}\nu_k(2^j,\Omega,\mu)^{\frac{p}{k+1}}}
\|u_{i,m}\|^p_{\pshko(\Om)}  \\
&\ge & C\frac{\sum_{j=i}^m2^{\frac{j(k+1)}{k+1-p}}\nu_k(2^j,\Omega,\mu)^{-\frac{p}{k+1-p}}}{\left(\sum_{j=i}^m2^{\frac{j(k+1)}{k+1-p}}\nu_k(2^j,\Omega,\mu)^{-\frac{p}{k+1-p}}\right)^{\frac{p}{k+1}}}
\|u_{i,m}\|^p_{\pshko(\Om)}  \\
&= & C\left(\sum_{j=i}^m2^{\frac{j(k+1)}{k+1-p}}\nu_k(2^j,\Omega,\mu)^{-\frac{p}{k+1-p}}\right)^{\frac{k+1-p}{k+1}} \|u_{i,m}\|^p_{\pshko(\Om)} .
\end{eqnarray*}
Consequently,
\begin{eqnarray*}
I_{k,p}(\Omega,\mu)&\le &  \lim_{i\to -\infty}\sum_{j=i}^m 2^{\frac{(j+1)(k+1)}{k+1-p}}\nu_k(2^j,\Omega,\mu)^{-\frac{p}{k+1-p}}+\int_{2^m}^{\mu(\Om)}  \left(\frac{s}{\nu_k(s,\Omega,\mu)}\right)^{\frac{p}{k+1-p}}\,ds\\
&\le & (1+\mu(\Om))\lim_{i\to -\infty}\sum_{j=i}^m 2^{\frac{(j+1)(k+1)}{k+1-p}}\nu_k(2^j,\Omega,\mu)^{-\frac{p}{k+1-p}}<+\infty.\qed
\end{eqnarray*}

 \subsection{The Sobolev type trace inequality(the case $p\ge k+1$).}
 First, we assume $I_{k,p}(\Omega,\mu)<+\infty$ and prove the Sobolev type inequality.  
For $u\in \mathcal F_k(\Om)$, we have
\begin{align*}
\int_{\Om}(-u)^p\,d\mu= &   p\int_0^{\infty}\mu(K_{t})t^{p-1}\,dt  \\
\le & pI_{k,p}(\Omega,\mu)\int_0^{\infty}{\ca_k(K_t, \Om)}^{\frac{p}{k+1}}t^{p-1}\,dt  \\
= & pI_{k,p}(\Omega, \mu)\int_0^{\infty}t^{p-1-k}t^k{\ca_k(K_t, \Om)}^{\frac{p-1-k}{k+1}+1}\,dt   \\
\le & pI_{k,p}(\Omega,\mu)\cdot\|u\|_{\pshko(\Om)}^{p-1-k}\int_0^{\infty}t^k\ca_k(K_t,\Om)\,dt   \\
\le & \left[p I_{k,p}(\Omega,\mu)\left(\frac{A}{A-1}\right)^{k+1}\log A\right]\cdot\|u\|_{\pshko(\Om)}^{p}.
\end{align*}
The sufficient part is proved.
\vskip 10pt

On the contrary, let $u=u_{k,K}^*$ the relatively capacitary potential with respect to a compact set $K\subset \Om$, then
$$\mu(K)^{\frac{1}{p}}\le \|u_{k,K}^*\|_{L^p(\Om)}\le C\cdot\ca_k(K,\Om)^{\frac{1}{k+1}}.$$ 
By Remark \ref{rem} (c), $I_{k,p}(\Omega,\mu)<+\infty$.

\subsection{Compactness}
In this section, we are going to consider the compactness of the embedding induced by inequalities \eqref{sobolev1}.
First, we recall the Poincar\'e type inequality for complex Hessian operators.
\begin{theo}\label{quotient1}\cite{Hou, AC20}
Suppose $\Om$ is a pseudoconvex domain with smooth boundary, and $1\le l< k\le n$. Then there exists a uniform constant $C>0$
depending on $k$, $l$ and $\Om$ such that
	\begin{align}
\|u\|_{\psh_{l,0}(\Om)}\le C\|u\|_{\pshko(\Om)},\ \ \ \forall u\in \pshko(\Om).
	\end{align}

\end{theo}
\begin{cor}\label{ccor}
	Suppose $K$ is a compact subset of $\Om$ with smooth boundary $\p K$, and $1\le l < k\le n$. Then we have 
\begin{align}\label{P}
	\ca_l(K,\Om)\le C\cdot {\ca_k(K,\Om)}^{\frac{l+1}{k+1}}.
\end{align}
\end{cor}
\begin{proof}
Let $u\in \pshko(\Om)$ such that $u\le -\chi_K$. By $\pshko(\Om)\subset \psh_{l,0}(\Om)$, we have
\begin{eqnarray*}
\ca_l(K,\Om)=\widetilde{\ca}_{l,3}(K,\Om)&\le& \int_\Om(-u)(dd^cu)^l\wedge\omega^{n-l}\\
&\le& C\left[\int_{\Om}(-u)(dd^cu)^{k}\wedge\omega^{n-k}\right]^{\frac{l+1}{k+1}}.
\end{eqnarray*}
Then \eqref{P} follows by taking the infimum. 
\end{proof}

Now, we are in position to show the compactness of the Sobolev trace inequality. 
First we need the following compactness theorem for classical Sobolev embedding. 
\begin{theo}\label{Maz}\cite{Ma}
Suppose $p\ge 2$, the embedding
$$Id:\,W^{1,2}_0(\Om)\to L^{p}(\Om,\mu)$$
is compact if and only if
\begin{align}\label{cond-com-1}
\lim_{s\to 0}\frac{s}{\nu_1(s,\mu)^{\frac{p}{2}}}\to 0.
\end{align}
\end{theo}
By Corollary \ref{ccor}, condition \eqref{cond-compact} implies \eqref{cond-com-1}.

\noindent{\it Proof of compactness in Theorem \ref{main3}(i).}
First, we consider the 'if' part.
For a sequence $u_j\in \mathcal F_k(\Omega)$ with bounded $\|u_j\|_{\pshko(\Omega)}$, we can obtain the boundedness of $\|u_j\|_{W^{1,2}_0(\Om)}$ by Theorem \ref{quotient1}.
Then by Theorem \ref{Maz}, we can find a subsequence $\{u_{j_l}\}$ such that  $u_{j_l}$ converges to $u$ in $L^{\frac{p+2}{2}}(\Om,\mu)$. Hence, we have
\begin{align}
\int_{\Om}|u_{j_l}-u|^p\,d\mu\le &  \left(\int_{\Om}|u_{j_l}-u|^{\frac{p+2}{2}}\,d\mu\right)^{\frac{1}{2}}\cdot \left(\int_{\Om}|u_{j_l}-u|^{\frac{3p-2}{2}}\,d\mu\right)^{\frac{1}{2}}  \nonumber\\
\le & C\|u_{j_l}-u\|_{L^{\frac{p+2}{2}}(\Om,\mu)}\cdot (\|u_{j_l}\|_{\pshko(\Om)}^{\frac{3p-2}{4}}+\|u\|^{\frac{3p-2}{4}}_{\pshko(\Om)})\to 0,  
\end{align}
thereby completing the proof.

For the 'only if' part, we assume $\pshko(\Om)$ is compactly embedded  into $L^p(\Om,\mu)$. Similar to the linear case \cite{Ma}, we will show that for any $s>0$, there exists $\eps(s)$ satisfying $\eps(s)\to 0$, as $s\to 0$, such that for any compact subset $K\subset\Om$ with $\mu(K)<s$,
it holds
\begin{align}\label{equi-norm}
\left(\int_K(-u)^p\,d\mu\right)^{\frac{1}{p}}\le \eps(s)\|u\|_{\pshko(\Om)}
\end{align}
holds for all $u\in \pshko(\Om)$.

We show this by contradiction. Assume there exists a sequence of compact subsets $K_i$ such that $\mu(K_i)\to 0$ and $\{u_j\}_{j=1}^{\infty}\subset \pshko(\Om)$, and a uniform constant $c>0$ such that
$$\int_{K_i}(-u_j)^p\,d\mu\ge c\|u_j\|_{\pshko(\Om)}^p.$$
By scaling, we may assume $\|u_j\|_{\pshko(\Om)}=1$. Then by the compactness of the embedding, there is a subsequence,   still denoted by $\{u_j\}$, which converges to $u\in L^p(\Om,\mu)$ in $L^p$-sense. Now we consider a sequence of cut-off functions $\eta_j\in C^{\infty}_0(\Om)$ such that $\text{supp}(\eta_j)\subset K_j$. Then by $\mu(K_i)\to 0$, $\eta_j^{\frac{1}{p}}u_j$ converges to $0$ in $L^p(\Om,\mu)$ sense. Hence, $\|u_j\|_{L^p(K_j,\mu)}\to 0$ as $j\to \infty$, which makes a contradiction.  

Finally, for any $K\Subset\Om $ with $\mu(K)<s$, by letting $u=u_{k,K}^*$ in the \eqref{equi-norm}, we have
$$\frac{\mu(K)}{\ca_k(K,\Om)^{\frac{p}{k+1}}}\le \eps(s)\to 0.\qed$$

\subsection{The Moser-Trudinger type trace inequality(Proof of Theorem \ref{main3}(ii))}
Let $\mu$ be a positive Randon measure,
\begin{eqnarray*}
&&\int_{\Om}\exp\left(\beta\left(\frac{-u}{\|u\|_{\psho(\Om)}}\right)^{q}\right)\,d\mu \\
&=&   \sum_{i=0}^{\infty}\frac{\beta^i}{i!}\int_{\Om}\left(\frac{-u}{\|u\|_{\psho(\Omega)}}\right)^{q i}\,d\mu  \\
&= & \sum_{i<\frac{n+1}{q}}\frac{\beta^i}{i!}\int_{\Om}\left(\frac{-u}{\|u\|_{\psho(\Omega)}}\right)^{q i}\,d\mu 
+\sum_{i\ge \frac{n+1}{q}}\frac{\beta^i}{i!}\int_{\Om}\left(\frac{-u}{\|u\|_{\psho(\Omega)}}\right)^{q i}\,d\mu   \\
&= :& I+II.
\end{eqnarray*}
Since $I_{n}(\beta,\Om,\mu)<+\infty$ implies $I_{k,p}(\Omega,\mu)$,  we have  $I\leq C$.
It suffices to estimate $II$. We have
\begin{align}\label{key-3}
II= & \sum_{i\ge \frac{n+1}{q}}\frac{\beta^i}{i!}\|u\|_{\psho(\Omega)}^{-q i}\int_{\Om}(-u)^{q i}\,d\mu   \nonumber\\
=  & \sum_{i\ge \frac{n+1}{q}}\frac{\beta^i}{i!}\|u\|_{\psho(\Omega)}^{-q i}\int_{0}^{\infty}t^{q i}\frac{d\mu(K_t)}{dt}\,dt  \nonumber\\
= & \sum_{i\ge \frac{n+1}{q}}\frac{\beta^i}{i!}\|u\|_{\psho(\Omega)}^{-q i}\int_0^{\infty}\mu(K_t)\,d(t^{q i})  \nonumber \\
\le & \sum_{i\ge \frac{n+1}{q}}\frac{\beta^i}{i!}\int_0^{\infty}\frac{\ca(K_t,\Om)}{t^{-n}\|u\|^{n+1}}\left(\frac{\mu(K_t)}{\ca(K_t,\Om)^{\frac{q i}{n+1}}}\right)\,dt    \nonumber\\
\le & \beta q \|u\|^{-n-1}\int_0^{\infty}\sum_{i=0}^{\infty}\frac{\beta^i}{i!}\left(\frac{\mu(K_t)}{\ca(K_t,\Om)^{\frac{q i}{n+1}}}\right)\,t^n \ca(K_t,\Om)\,dt \nonumber \\
= & \beta q \|u\|^{-n-1}\int_0^{\infty}\mu(K_t)\exp\left(\frac{\beta}{\ca(K_t,\Om)^{\frac{q}{n+1}}}\right)t^n \ca(K_t,\Om)\, dt, \nonumber \\
\le & \beta q I_{n}(\beta,\Omega,\mu)\left(\frac{A}{A-1}\right)^{k+1}\log A,\nonumber
\end{align}
where we have used \eqref{w-ineq} at the fourth line and \eqref{s-ineq} at the last line. \qed

\begin{rem}\label{rem1}
When the measure $d\mu$ is Lebesgue measure and $\Om=B_1$ is the unit ball centered at the origin,
as proved in \cite{K96}, there holds 
\beq\label{isomoser}
|K|\le C_{\lambda,\Omega,n}\exp\left\{-\frac{\lambda}{\ca_n(K,\Om)^{\frac{1}{n}}}\right\}
\eeq
for any $0<\lam<2n$.	
In particular, when $K=B_r$ with $r<1$, by standard computations,
$$|B_r|e^{\frac{\lambda}{\ca_n(B_r,\Om)^{\frac{1}{n}}}}=C_nr^{2n-\lam}.$$
Hence \eqref{isomoser} does not hold when $\lam>2n$. 
It is  natural to ask  if $\lam$ can attain the optimal constant $2n$.

As shown in \cite{BB}, for those $u\in \psho(\Om)$ with $\bS^1$-symmetry, i.e. $u(z)=u(e^{\sq \theta}z)$ for all $\theta\in\R^1$,
 \eqref{MT} holds with $\beta=2n$. Then by the proof of Theorem \ref{main3} (ii) with $\bS^1$-symmetry, there holds
$$|E|\le C\exp\left\{-\frac{2n}{\ca_n(E,\Om)^{\frac{1}{n}}}\right\}$$
for any $\bS^1$-invariant subset $E$.
Moreover,
by \cite{BB14}, the Schwartz symmetrization $\hat u$ of  $u_{E}^*$ is plurisubharmonic and has smaller $\|\cdot\|_{\psho(
\Omega)}$-norm. This leads to 
$$|E|e^{\frac{2n}{\ca_n(E,\Om)^{\frac{1}{n}}}}\le |B_r|e^{\frac{2n}{\ca_n(B_r,\Om)^{\frac{1}{n}}}}\equiv C_n|B_1|.$$

\end{rem}

\subsection{The Brezis-Merle type trace inequality}
Similar to \cite{BB}, we can obtain a relationship between the Brezis-Merle type trace inequality and the Moser-Trudinger type inequality.

\begin{theo}\label{remB}
The Moser-Trudinger type trace inequality \eqref{MT} holds for any $0<\lam<\beta$ 
for some  $\beta>0$ and $q\in [1,\frac{n+1}{n}]$ if and only if for any $0<\lam<\beta$, the Brezis-Merle type trace inequality
\begin{align}\label{BM}
\sup\left\{\int_{\Om}\exp\left(\lam\left(\frac{-u}{\cM[u]^{\frac{1}{n}}}\right)^{\frac{nq}{n+1}}\right)\,d\mu: u\in \mathcal F_n(\Omega), 0<\|u\|_{\psho(\Om)}<\infty\right\}<+\infty
\end{align}
holds.
Here $\cM[u]:=\int_{\Om}(dd^cu)^n$.
\end{theo}

\begin{proof}
First, we prove the if part. By Theorem \ref{main3}(iv), there is $C>0$ such that for any compact subset $K\subset \Om$ $$\mu(K)e^{\lam\frac{1}{\ca_n(K,\Om)^{\frac{q}{n+1}}}}\le C.$$
For $u\in\F_k(\Om)$, we denote $K_t:=\{u\le -t\}$, $t>0$. By comparison principle, 
we have 
$$\cM[u]=\int_{\Om}(dd^cu)^n\ge t^n\int_{\Om}(dd^cu_{K_t})^n=t^n\ca_n(K_t,\Om), \ \forall t>0.$$  
Therefore, 
$$\mu(K_t)\le Ce^{-\lam\frac{1}{\ca_n(K,\Om)^{\frac{q}{n+1}}}}\le Ce^{-\lam\frac{t^{\frac{nq}{n+1}}}{\cM[u]^{\frac{q}{n+1}}}}.$$
Then for every  $0<\eps<\lam$, 
\begin{eqnarray*}
\int_{\Om}e^{(\lam-\eps)\frac{(-u)^{\frac{nq}{n+1}}}{\cM[u]^{\frac{q}{n+1}}}}\,d\mu
&=& \sum_{j=0}^{\infty}\frac{(\lam-\eps)^j}{j!}\int_{\Om}\frac{(-u)^{\frac{nqj}{n+1}}}{\cM[u]^{\frac{qj}{n+1}}}\,d\mu  \\
&= & \sum_{j=0}^{\infty}\frac{(\lam-\eps)^j}{j!}\int_0^{\infty}\frac{\mu(K_t)}{\cM[u]^{\frac{qj}{n+1}}}\,d\left(t^{\frac{nqj}{n+1}}\right)  \\
&\le & C\sum_0^{\infty}\frac{(\lam-\eps)^j}{j!}\int_0^{\infty}e^{-\lam\frac{t^{\frac{nq}{n+1}}}{\cM[u]^{\frac{q}{n+1}}}}d\left(\frac{t^{\frac{nqj}{n+1}}}{\cM[u]^{\frac{qj}{n+1}}}\right)  \\
&\le & \sum_{j=0}^{\infty}\frac{(\lam-\eps)^j}{\lam^j}=\frac{\lam}{\eps}-1.
\end{eqnarray*}

Next, we show the only if part. For a Borel subset $E\subset \Om$, we can apply the Brezis-Merle trace inequality to $u_E^*$ to obtain the condition \eqref{cond-MT1}, which implies the Moser-Trudinger type inequality.
\end{proof}

\vskip 30pt

\section{The Dirichlet problem}

\subsection{The continuous solution(Proof of Theorem \ref{main5})}
First, we prove (ii) $\Rightarrow$ (i).
Suppose $E\subset \Om$ is a Borel subset, and $u$ solves
\begin{align}\label{diri1}
\begin{cases}
(dd^cu)^k\wedge\omega^{n-k}=\chi_E\,d\mu,\ \ \ &\text{in }\Omega,  \\
u=0,  &\text{on }\p \Om,
\end{cases}
\end{align}
where $\mu$ is a positive Radon measure.
\begin{theo}\label{GPT-infty}
Assume there exists $p>k+1$ such that  
$I_{k,p}(\Om,\mu)<+\infty$.
Let $u$ be a solution to \eqref{diri1}.
Then there exists $C_1>0$ depending on $k$, $p$ and $I_{k,p}(\Om,\mu)$ such that
	\beq\label{inftyest}
	\|u\|_{L^{\infty}(\Omega,\mu)}\le C_1\mu(E)^{\delta}
	\eeq
	where $\delta=\frac{p-k-1}{kp}$.
\end{theo}
\begin{proof}
The proof is similar to \cite{WWZ2}.
	We denote $\hat \mu:=\chi_{E}\cdot\mu$. For any $s>0$, let
$K_s:=\{u\le -s\}$ and $u_s=u+s$.  Denote $p=\frac{k+1}{1-k\delta}$. By $I_{k,p}(\Om,\mu)<+\infty$, 
we can apply Theorem \ref{main3} to get 
$$\left(\int_{K_s}\left(\frac{-u _s}{\|u _s\|_{\pshko(K_s)}}\right)^pd\hat\mu \right)^{\frac{1}{p}}\leq C.$$
Then by the equation, 
\begin{eqnarray*}
t\int_{K_{s+t}}(dd^cu _s)^k\wedge\omega^{n-k} & \le  & \int_{K_s}(-u _s)\,d\hat\mu  \\
&\le& \left(\int_{K_s}(-u _s)^pd\hat\mu \right)^{\frac{1}{p}}\left(\int_{K_s}d\hat\mu \right)^{1-\frac{1}{p}}   \\
&=&\left(\int_{K_s}\left(\frac{-u _s}{\|u _s\|_{\pshko(K_s)}}\right)^pd\hat\mu \right)^{\frac{1}{p}}\left(\int_{K_s}d\hat\mu \right)^{1-\frac{1}{p}}\cdot \|u _s\|_{\pshko(K_s)},
\end{eqnarray*} 
which implies 
\begin{align}\label{ite}
t\hat\mu (K_{s+t})\le C\hat\mu (K_s)^{(1-\frac{1}{p})\frac{k+1}{k}}=C\hat\mu (K_s)^{1+\delta}.
\end{align}
In particular,
$$\hat\mu (K_s)\le \frac{C}{s}\hat\mu (\Om)^{1+\delta}$$
for some $C>1$.
Then choose $s_0=2^{1+\frac{1}{\delta }}C^{1+\frac{1}{\delta }}\hat\mu (\Om)^{\delta }$, we get $\hat\mu (K_{s_0})\leq \frac{1}{2}\hat\mu (\Om)$. For any $l\in \mathbb{Z}_+$, define 
\begin{align}\label{iteration}
s_l=s_0+\sum_{j=1}^l2^{-\delta  j}\hat\mu (\Om)^{\delta },\ \ u^l=u+s^l, \ \ K_l=K_{s_l}.
\end{align}
Then
$$2^{-\delta  (l+1)}\hat\mu (\Om)^{\delta }\hat\mu (K_{l+1})=(s_{l+1}-s_l)\hat\mu (K_{l+1})\le C\hat\mu (K_l)^{1+\delta}.$$

We claim that $|K_{l+1}|\leq \frac{1}{2}|K_l|$ for any $l$. By induction, we assume the inequality holds for $l\leq m$. 
Then
\begin{align*}
\hat\mu (K_{m+1})\leq &  C\hat\mu (K_m)^{1+\delta }\frac{2^{\delta  (m+1)}}{\hat\mu (\Om)^{\delta }}    \\
\leq &  C\left[\left(\frac{\hat\mu (K_0)}{2^m}\right)^{\delta }\frac{2^{\delta  (m+1)}}{\hat\mu (\Om)^{\delta }}\right]\cdot |K_m|  \\
\leq & \left[C^{1+\delta }\frac{2^{\delta }}{s_0^{\delta }}\hat\mu (\Om)^{\delta ^2}\right]\cdot \hat\mu (K_m)   
\leq \frac{1}{2}\hat\mu (K_m). 
\end{align*}
This implies that the set 
$$\hat\mu \left(\left\{x\in \Omega\big|u<-s_0-\sum_{j=1}^{\infty}\left(\frac{1}{2^{\delta}}\right)^j\hat\mu (\Omega)^{\delta }\right\}\right)=0.$$
Hence, 
\beqs
\|u\|_{L^{\infty}(\Omega, \mu)}&\leq& s_0+\sum_{j=1}^{\infty}\left(\frac{1}{2^{\delta }}\right)^j\hat\mu (\Omega)^{\delta }\\
&=&2^{1+\frac{1}{\delta }}C^{1+\frac{1}{\delta }}\hat\mu (\Omega)^{\delta }+\frac{1}{2^{\delta }-1}\hat\mu (\Omega)^{\delta }\\
&\leq& C\hat\mu (\Omega)^{\delta }=C\mu(E)^{\delta}.
\eeqs
\vskip -30pt
\end{proof}

Next, by similar arguments as in \cite{WWZ2}, we can get a stability result as well as the existence of the unique continuous solution $u_E$ for \eqref{k-Dir}. We need the stability lemma.
\begin{lem}\label{stab-holder-mu}
	Let $u$, $v$ be  bounded $k$-plurisubharmonic functions in $\Omega$ satisfying $u\geq v$ on $\p \Omega$. 
	Assume $(dd^cu)^k\wedge\omega^{n-k}=d\mu$ and $\mu$ satisfies the condition in Theorem \ref{main5} (ii).
	Then   $\forall\ \eps>0$, there exists $C>0$ depending on $k$, $p$ and $I_{k,p}(\Om,\mu)$, 
	\beq\label{vu2-mu}
	\sup\limits_{\Omega}(v-u)\leq \eps+C\mu(\{v-u>\eps\})^{\delta}.
	\eeq
\end{lem}

\begin{proof}
We may suppose $d\mu\in L^\infty(\Omega)$ for the estimate since we will consider the approximation of $\mu$ later in the proof of existence and
continuity.
	Denote $u_{\eps}:=u+\eps$ and $\Omega_{\eps}:=\{v-u_{\eps}>0\}$.
	It suffices to estimate $\sup_{\Omega_{\eps}}|u_{\eps}-v|$.
	
	Note that $\Omega_{\eps}\Subset\Omega$ and $u_{\eps}$ solves 
	$$
	\begin{cases}
		 (dd^cu_{\eps})^k\wedge\omega^{n-k}=d\mu \  \  &\text{\ in $\Omega_{\eps}$,} \\[-3pt]
		u_{\eps}=v           & \text{\ on $\pom_{\eps}$.}
		\end{cases}
	$$
	Let
	$u_0$ be the solution to the Dirichlet problem
	$$
	\begin{cases}
		(dd^cu_{0})^k\wedge\omega^{n-k}=\chi_{\Om_{\eps}}d\mu \ \ &\text{\ in $\Omega$,} \\[-5pt]
		 \ u_{0}=0           \hskip55pt      \ \ &\text{\ on $\pom$.}
		\end{cases}
	$$
	By the comparison principle we have 
	$$u_0\leq u_{\eps}-v\leq 0 \ \ \text{in\ $\Omega_{\eps}$}.$$ 
	Hence we obtain
	$$\sup_{\Omega_{\eps}}|u_{\eps}-v|
	\leq \sup_{\Omega_{\eps}}|u_0|\leq C\mu(\Om_\eps)^{\delta}.$$
\end{proof}

\begin{prop}\label{prep-holder-mu}
	Let $u$, $v$ be bounded $k$-plurisubharmonic functions in $\Omega$ 
	satisfying $u\geq v$ on $\p \Omega$. 
	Assume that $(dd^cu)^k\wedge\omega^{n-k}=d\mu$ and $\mu$ satisfies  the condition in Theorem \ref{main5} (ii).
	Then for $r\geq 1$ and $0\leq \gamma'<\frac{\delta r}{1+\delta r}$, it holds 
	\beq\label{vu-mu}
	\sup\limits_{\Omega}(v-u)\leq C\|\max(v-u,0)\|_{L^r(\Omega,d\mu)}^{\gamma}
	\eeq
	for a uniform constant $C=C(\gamma',\|v\|_{L^{\infty}(\Omega,d\mu)})>0$. 
\end{prop}

\begin{proof} Note that for any $\eps>0$,
	$$\mu(\{v-u>\eps\})
	\leq  \eps^{-r}\int_{\{v-u>\eps\}}|v-u|^{r} \,d\mu  \leq \eps^{-r}\int_\Omega[\max(v-u,0)]^r\,d\mu.$$
	Let $\eps:=\|\max(v-u,0)\|_{L^r(\Omega,d\mu)}^{\gamma'}$, 
	where $\gamma'$ is to be determined. 
	By Lemma \ref{stab-holder-mu}, we have
	\begin{eqnarray}
	\sup\limits_{\Omega}(v-u)
	&\leq & \eps+C\mu(\{v-u>\eps\})^{\delta} \label{vu3-mu} \\
	&\leq & \|\max(v-u,0)\|_{L^r(\Omega,d\mu)}^{\gamma'}+C\|\max(v-u,0)\|_{L^r(\Omega,d\mu)}^{\delta(r-\gamma' r)}.  \nonumber
	\end{eqnarray}
	Choose $\gamma'\leq \frac{\delta r}{1+\delta r}$, where $0<\delta<\frac{1}{k}$,
	\eqref{vu-mu} follows from \eqref{vu3-mu}.
\end{proof}

Now we can show the existence of the unique continuous solution for \eqref{k-Dir} when $\mu$ satisfies 
 the condition in Theorem \ref{main5} (ii).  We consider
$\tilde\mu_{\eps}:=\rho_{\eps} *d\mu$ defined on $\Om_{\eps}:=\{x\in \Om\,|\,\text{dist}(x,\p \Om)\le \eps\}$,
where $\eps>0$ and  $\rho_{\eps}$ is the cut-off function such that 
$$\rho_{\eps}\in C^{\infty}_0(\R^n),\  \rho_{\eps}=1\ \ \text{in }B_{\eps}(O),\  \text{and }\rho_{\eps}=0\text{ on }\R^n\setminus \overline{B_{2\eps}(O)}.$$
Then for every $E\subset \Om$, we define
$$\mu_{\eps}(E):=\tilde{\mu}_{\eps}(E\cap \Om_{\eps}).$$
By the classic measure theory, $\mu_{\eps}$ is absolutely continuous with respect to Lebesgue measure $\omega^n$ 
 	with bounded density functions $f_{\eps}$. We are going to check $\mu_{\eps}$ satisfies  the condition in Theorem \ref{main5} (ii) uniformly. 
 	Denote $K_y:=\{x\in\Om:x-y\in K\cap \Om_{\eps}\}$. By definition, for every compact subset $K\subset \Om$, we have
\begin{align}\label{11}
\mu_{\eps}(K)=\int_{\C^n}\int_{K_y}\rho_{\eps}(x-y)\,d\mu(x)\,\omega^n(y)  
\le  C\sup_{|y|\le 2\eps}\mu(K_y)  
\le  C\sup_{|y|\le 2\eps}\left(\ca_{k}(K_y,\Om)\right)^{p}.
\end{align}
 Let $u_y(x):=u_{k,K_y}^*(x+y)$, $u_{\eps}(x):=u_{k,\Om_{\eps}}^*(x)$ be the relative $k$-extreme functions of $K_y$, $\Om_{\eps}$ with respect to $\Om$.
 For any $0<c<\frac{1}{2}$, denote $\Om_{c}=\{u_{\eps}\le -c\}$.
Let
$$
g(x):=\begin{cases}
\max\left\{u_y(x)-c, (1+2c)u_{\eps}\right\},\ \ \ &x\in \Om_{\frac{c}{2}};  \\
(1+2c)u_{\eps},   &x\in \Om\setminus\Om_{\frac{c}{2}}.
\end{cases}
$$
Note that $g$ is a well-defined $k$-plurissubharmonic function in $\Om$. Since $K_y\subset \Om_{\eps}$ for every $|y|\le \eps$, we have $u_y-c\ge (1+2c)u_{\eps}=-1-2c$ on $\overline{\Om_{\eps}}$. Hence for any compact subset $K\subset \Om$, we can get
\begin{eqnarray}\label{22}
\ca_k(K,\Om)&\ge& (1+2c)^{-k}\int_{K\bigcap \Om_{\eps}}(dd^cg)^k\wedge\omega^{n-k}=(1+2c)^{-k}\int_{K\bigcap \Om_{\eps}}(dd^cu_y)^k\wedge\omega^{n-k}  \nonumber\\
&=&(1+2c)^{-k}\int_{K_y}(dd^cu_{k,K_y}^*)^k\wedge\omega^{n-k}=(1+2c)^{-k}\ca_k(K_y,\Om).
\end{eqnarray}
By \eqref{11} and \eqref{22}, we have shown that $\mu_{\eps}$ satisfies  the condition in Theorem \ref{main5} (ii) uniformly.

 Hence the solutions $\{u_{\eps}\}$ to \eqref{k-Dir} with $\{f_\eps\}$ are uniformly bounded and continuous. Furthermore,
 there exists a sequence $\eps_j\to 0$ such that $\lim_{j\to \infty}\|\mu_{\eps_j}-\mu\|=0$, where $\|\cdot\|$ denotes 
 the total variation of a signed measure.
 Then the proof is finished by the following well-known result.

\begin{prop}
Let $\mu_j=(dd^c\varphi_j)^k\wedge\omega^{n-k}$, $\mu=(dd^c\varphi)^k\wedge\omega^{n-k}$ be 
non-pluripolar nonnegative Radon measures with finite mass, where $\varphi_j$, $\varphi\in\pshko(\Om)$ and 
$\|\varphi_j\|_{L^\infty(\Omega,d\mu)}$, $\|\varphi\|_{L^\infty(\Omega, d\mu)}\le C$. If $\|\mu_j-\mu\|\to 0$, then
$$\varphi_j\to\varphi\text{ in }L^1(\Om,d\mu).$$
\end{prop}
\begin{proof}
When $k=n$, this is Proposition 12.17 in \cite{GZ}. The proof also applies to $k<n$. We write a sketch of its proof here.
Note that by $\|\mu_j-\mu\|\to 0$, the measure
$$\nu:=2^{-1}\mu+\sum_{j\ge 2}2^{-j}\mu_j$$
is a well-defined non-pluripolar nonnegative Radon measure.  $\mu$, $\mu_j$ are absolutely continuous with respect to $\nu$. We may suppose that $\mu_j=f_j\nu$, $\mu=f\nu$ and $f_j\to f$ in $L^1(\Om,\nu)$. Then by  the weak compactness, there exists a subsequence $\varphi_j$ and $\psi\in\pshko(\Om)$ such that $\varphi_j\to \psi$ in $L^1(\Om,d\mu)$. Denote 
$\psi_j=\left(\sup_{l\ge j}\varphi_l\right)^*$, which converges to $\psi$ decreasingly almost everywhere with respect to $\mu$.
Then by comparison principle, we have
$$(dd^c\psi_j)^k\wedge\omega^{n-k}\le (dd^c\varphi_j)^k\wedge\omega^{n-k}=d\mu_j,$$
which implies $(dd^c\psi)^k\wedge\omega^{n-k}\le d\mu$. To get the equality, we use the absolutely continuity to obtain
$$(dd^c\psi_j)^k\wedge\omega^{n-k}\ge \inf_{l\ge j}f_l\,d\nu,$$
thereby completing the proof.
\end{proof}

The proof of (i) $\Rightarrow$ (ii) is simple. Let $\varphi_E$ be the solution to \eqref{k-Dir}. Then $\hat \varphi:=\frac{\varphi_E}{C\mu(E)^{\delta}}\in\pshk(\Om)$, and $-1\le \hat \varphi\le 0$. By definition we have
$$C^{-k}\mu(E)^{1-k\delta}=\frac{1}{C^k\mu(E)^{k\delta}}\int_E(dd^c\varphi_E)^k\wedge\omega^{n-k}=\int_{E}(dd^c\hat \varphi)^k\wedge\omega^{n-k}\le \ca_k(E,\Om).$$
In view of \eqref{cond-sobolev1},
we have completed the proof.

\subsection{H\"older continuity(Proof of Theorem \ref{main6})}
In this section, we consider the H\"older continuity of the solution. As we mentioned in the introduction, we will give a pure PDE proof for Theorem \ref{main6} as in \cite{WWZ2}.
For the proof from (i) to (ii), we just follow Proposition 2.4 in \cite{DKN}, which is a PDE approach.  It suffices to consider the other direction. Suppose \eqref{cond-DKN} holds.
In order to obtain the H\"older estimate, we need an $L^{\infty}$-estimate like \eqref{inftyest} with the measure replaced by Lebesgue measure. 
We suppose the solution $u$ and the measure $d\mu$ in \eqref{cMA} are smooth, and the general result follows by approximation. 

First, we  show there is an upper bound on $\|u\|_*$ under condition \eqref{cond-DKN}.
Note that by the classical Sobolev inequality and Theorem \ref{quotient1}, we have
\begin{align}\label{*1}
\|u\|_*\le C\|u\|_{L^1(\Om,\omega^n)}+C\left(\int_{\Om}du\wedge d^cu\wedge\omega^{n-1}\right)^{\frac{1}{2}}\le C\cE_k(u)^{\frac{1}{k+1}}<+\infty.
\end{align}
By condition \eqref{cond-DKN}, we get
\begin{align}
\frac{\cE_k(u)}{\|u\|_*}=\int_{\Om}\left(-\frac{u}{\|u\|_*}\right)\,d\mu\le C\frac{\|u\|_{L^1(\Om,\omega^n)}^{\gamma}}{\|u\|_{*}^{\gamma}}.
\end{align}
This implies
\begin{align}\label{*2}
\cE_k(u)^{1-\frac{\gamma}{k+1}}\le C\|u\|_*^{1-\gamma}.
\end{align}
Then by \eqref{*1} we get $\|u\|_*\le c_0$.

Choose $f=\hat u:=\frac{u}{c_0}$ in condition \eqref{cond-DKN}. Then we get
$$\int_{\Om}(-\hu)\,d\mu\le C\left(\int_{\Om}(-\hu)\,\omega^n\right)^{\gamma}.$$
Then in the proof of the $L^{\infty}$-estimate(Theorem \ref{GPT-infty}), we use a different iteration. 
We denote $\Om_s:=\{\hu\le -s\}$, $p< \frac{n(k+1)}{n-k}$(when $k=n$, we can choose any $p>0$)  and $\hu_s:=\hu+s$.
Note that
\begin{align}\label{energy*}
\cE_{k,\Om_s}(\hu_s):= & \int_{\Om_s}(-\hu_s)\,d\mu \le C\left(\int_{\Om_s}(-\hu_s)\,\omega^n\right)^{\gamma}  \nonumber\\
\le & C\left(\int_{\Om_s}\left(\frac{-\hu_s}{\cE_{k,\Om_s}(\hu_s)^{\frac{1}{k+1}}}\right)^p\,\omega^n\right)^{\frac{\gamma}{p}}\cE_{k,\Om_s}(\hu_s)^{\frac{\gamma}{k+1}}|\Om_s|^{(1-\frac{1}{p})\gamma}.
\end{align}
By the Sobolev inequality for the complex Hessian equation with respective to the Lebesgue measure \cite{AC20}, we get
\begin{eqnarray*}
t|\Om_{s+t}|\le  \int_{\Om_s}(-\hu_s)\,\omega^n&\le& C\left(\int_{\Om_s}(-\hu_s)^p\,\omega^n\right)^{\frac{1}{p}}|\Om_s|^{1-\frac{1}{p}}\\
&\le& C\cE_{\Om_s}(\hu_s)^{\frac{1}{k+1}} |\Om_s|^{1-\frac{1}{p}}\le C|\Om_s|^{1+\frac{\gamma p-k-1}{(k+1-\gamma)p}}.
\end{eqnarray*}
Hence,
$$\|\hu\|_{L^{\infty}(\Om,\omega^n)}\le C|\Om|^{\frac{\gamma p-k-1}{(k+1-\gamma)p}}.$$
By \eqref{*1} and \eqref{energy*}, we have $\|u\|_*\le C|\Om|^{\frac{(p-1)\gamma}{(k+1-\gamma)p}}$. 
Therefore, for $1<p\le \frac{n(k+1)}{n-k}$ when $1\le k<n$ and $p>1$ when  $k=n$,
\beq\label{maes}
\|u\|_{L^{\infty}(\Om,\omega^n)}\le C|\Om|^{\frac{\gamma p-k-1}{(k+1-\gamma)p}}\|u\|_*\le C|\Om|^{\frac{2\gamma p-\gamma-k-1}{(k+1-\gamma)p}},
\eeq
where  
\beq\label{del}
\delta:=\frac{2\gamma p-\gamma-k-1}{(k+1-\gamma)p}.
\eeq

Then by the same arguments as in \cite{WWZ2}, we can obtain the H\"older continuity.
For readers' convenience and the self-containness of this paper, we include a sketch here.

First, since we have $L^{\infty}$-estimate \eqref{maes} now, we can get the similar stability results as in Lemma \ref{stab-holder-mu} and Proposition \ref{prep-holder-mu}, in which we change the measure $d\mu$ to be the standard Lebesgue measure $\omega^n$. 

\begin{prop}\label{prep-holder}
	Let $u$, $v$ be bounded $k$-plurisubharmonic functions in $\Omega$ 
	satisfying $u\geq v$ on $\p \Omega$. 
	Assume that $(dd^cu)^k\wedge\omega^{n-k}=d\mu$ and $d\mu$ satisfies the \eqref{cond-DKN}.
	Then for $r\geq 1$ and $0\leq \gamma'<\frac{\delta r}{1+\delta r}$, it holds 
	\beq\label{vu}
	\sup\limits_{\Omega}(v-u)\leq C\|\max(v-u,0)\|_{L^r(\Omega,\omega^n)}^{\gamma}
	\eeq
	for a uniform constant $C=C(\gamma',\|v\|_{L^{\infty}(\Om,\omega^n)})>0$. Here $\delta$ is given by \eqref{del}.
\end{prop}


\vskip 10pt

For any $\eps>0$, we denote $\Omega_\eps:=\{x\in \Omega|\, dist(x,\pom)>\eps\}$.
Let
\begin{eqnarray*}
	u_{\eps}(x)&:=&\sup\limits_{|\zeta|\leq \eps}u(x+\zeta),\ x\in\Omega_{\eps},\\
	\hat{u}_{\eps}(x)&:=&\bbint_{|\zeta-x|\leq \eps}u(\zeta)\,\omega^n,\ x\in\Omega_{\eps}.
\end{eqnarray*}
Since $u$ is plurisubharmonic in $\Omega$,  $u_\eps$ is a  plurisubharmonic function.
For the H\"older estimate,  it suffices to show there is a uniform constant $C>0$ such that $u_{\eps}-u\leq C\eps^{\alpha'}$ for some $\alpha'>0$. 
The link between $u_{\eps}$ and $\hat{u}_{\eps}$ is made by the following lemma.  
\begin{lem}(Lemma 4.2 in \cite{Ngu})\label{interchange}
	Given $\alpha\in (0,1)$, the following two conditions are equivalent. 
	
	(1) There exists $\eps_0$, $A>0$ such that for any $0<\eps\leq \eps_0$, 
	$$u_{\eps}-u\leq A\eps^\alpha\ \ \text{on\ $\Omega_\eps$}.$$
	
	(2) There exists $\eps_1$, $B>0$ such that for any $0<\eps\leq \eps_1$, 
	$$\hat{u}_{\eps}-u\leq B\eps^\alpha\ \ \text{on\ $\Omega_\eps$}.$$
\end{lem}

The following estimate is a generalization of Lemma 4.3 in \cite{Ngu}.

\begin{lem}\label{lapalace-control}
	Assume $u\in W^{2, r}(\Omega)$ with $r\geq 1$. Then
	for $\eps>0$ small enough, we have 
	\begin{equation} \label{L54}
	\left[\int_{\Omega_{\eps}}|\hat{u}_{\eps}-u|^r\,\omega^n\right]^{\frac{1}{r}}\leq C(n,r)\|\triangle u\|_{L^r(\Omega,\omega^n)}\eps^2
	\end{equation}
	where $C(n,r)>0$ is a uniform constant. 
\end{lem}

Note that the function $u_\eps$ is not globally defined on $\Omega$. 
However, by $\varphi\in C^{2\alpha}(\partial\Omega)$, 
there exist $k$-plurisubharmonic functions $\{\tilde u_\eps\}$ which decreases to $u$
as $\eps\to 0$ and satisfies \cite{GKZ}
\begin{equation}\label{barrier}
\begin{cases}
\tilde u_\eps=u+C\eps^\alpha & \text{in}\ \Omega\setminus \Omega_\eps,\\[4pt]
\hat u_\eps\leq\tilde u_\eps\leq \hat u_\eps+C\eps^\alpha& \text{in}\ \Omega_\eps ,
\end{cases}
\end{equation}
where the constant $C$ is independent of $\eps$. 
Then if $u\in W^{2, r}(\Omega)$, by choosing $v=\tilde{u}_{\eps}$,  
$\gamma'< \frac{\delta r}{1+\delta r}$ in Proposition \ref{prep-holder}, and using Lemma \ref{lapalace-control}, we have 
\begin{eqnarray}\label{HR}
\sup_{\Omega_\eps} (\hat u_\eps-u)&\leq& \sup_\Omega(\tilde u_\eps-u)+C\eps^\alpha \nonumber\\
&\leq& C\|\tilde u_\eps-u\|_{L^r}^{\gamma'}+C\eps^\alpha\\
&\leq & C\|\triangle u\|_{L^r(\Omega,\omega^n)}^{\gamma'}\eps^{2\gamma'}+C\eps^\alpha\nonumber.
\end{eqnarray}
Hence, once we have $u\in W^{2, r}$ for $r\geq 1$, it holds $u\in C^{\gamma'}$ for $\gamma'<\frac{\delta r}{1+\delta r}$, where $\delta=\frac{2\gamma p-\gamma-k-1}{(k+1-\gamma)p}$ and $p\le \frac{n(k+1)}{n-k}$($p\ge 1$ when $k=n$).

Finally, we show that under the assumption of Theorem \ref{main6}, it holds $u\in W^{2,1}(\Omega)$, i.e., $\triangle u$ has finite mass, and hence $u\in C^{\gamma'}$ for $\gamma'<\frac{\delta r}{1+\delta r}$. 
Let $\rho$ be the smooth solution to 
\begin{align*}
\begin{cases}
(dd^c\rho)^k\wedge\omega^{n-k}=\omega^n,\ \     &\text{in}\ \Om,  \\[4pt]
\rho=0,\ \    &\text{on}\ \p \Om.
\end{cases}
\end{align*}
We can choose $A>>1$ such that $dd^c(A\rho)\ge \omega$. Then by the generalized  Cauchy-Schwartz inequality,
\begin{eqnarray*}
	\int_{\Omega}dd^cu\wedge \omega^{n-1}
	&\le& \int_{\Omega}dd^cu\wedge [dd^c(A\rho)]^{n-1}\\
	&\le& A^{k-1}\left(\int_{\Om}(dd^c\rho)^k\wedge\omega^{n-k}\right)\cdot\left(\int_{\Om}(dd^cu)^k\wedge\omega^{n-k}\right)\leq C.
\end{eqnarray*}
Hence we finish the proof.

\begin{ex}
As an example, we consider the Dirichlet problem \eqref{k-Dir} with the measure $\mu_S$ defined by $\mu_S(E)=\cH^{2n-1}(E\cap S)$ for a real hypersurface $S\subset \C^{n}$ and $E\subset \Om$, where $\cH^{2n-1}(\cdot)$ is the $2n-1$-Hausdorff measure. By Stein-Tomas Restriction Theorem,
for all $f\in C^{\infty}_0(\C^n)$, there holds \cite{BS} $$\|f|_S\|_{L^2(S, \mu_S)}\le C\|f\|_{L^p(\Om)},\ \ \ 1\le p\le \frac{4n+2}{2n+3}.$$
Therefore, we can apply Theorem \ref{main6} with $\gamma=1$. This leads to 
$u\in C^{\gamma'}(\Om)$ for $$\gamma'<\frac{n+k+2}{(n+1)(k+2)}.$$
\end{ex}

\end{document}